\documentclass{amsart}
\usepackage[utf8]{inputenc}
\usepackage[T1]{fontenc}
\usepackage{lmodern}
\usepackage{verbatim}
\usepackage{amssymb,amscd,amsmath,amsthm}
\usepackage{mathtools}
\usepackage[margin=1.2in]{geometry}
\usepackage{indentfirst}
\usepackage{enumerate}
\usepackage{cases}
  
\numberwithin{equation}{section}
\newtheorem{thm}{Theorem}[section]
\newtheorem{lem}[thm]{Lemma}
\newtheorem{prop}[thm]{Proposition}
\newtheorem{cor}[thm]{Corollary}
\newtheorem{dfn}{Definition}[section]
\newtheorem{rmk}[thm]{Remark}

\newcommand{\R}{\mathbb{R}}
\newcommand{\Z}{\mathbb{Z}}

\title{On Bourgain's bound for short exponential sums and squarefree numbers}
\author{Ramon M. Nunes}
\address{Universit\'e Paris Sud, Laboratoire de math\'ematiques\\
Campus d'Orsay\\ 91405 Orsay Cedex\\ France}
\email{ramon.moreira@math.u-psud.fr}

\date{\today}

\begin{document}

\begin{abstract}
We use Bourgain's recent bound for short exponential sums to prove certain independence results related to the  distribution of squarefree numbers in arithmetic progressions.
\end{abstract}

\maketitle


\section{Introduction}

As usual, let

$$
e(x):=e^{2i\pi x},\text{ for $x\in \R$}.
$$
In a recent paper, Bourgain \cite{Bourgain} proved a non trivial bound for exponential sums such as
$$
\sum_{\substack{n\leq N\\ (n,q)=1}}e\left(\frac{a\overline{n}^2}{q}\right),
$$
where $q>1$ is an integer and ${\bar n}$ denotes the multiplicative inverse of $n\!\!\pmod q$, in the range $N\geq q^{\epsilon}$, for an arbitrarily small, but fixed, $\epsilon>0$. In his paper, Bourgain was interested in an application related to the size of fundamental solutions $\epsilon_D>1$ to the Pell equation
$$
t^2-Du^2=1.
$$
He followed the lead of Fouvry \cite{Fouvry}, who suggested that such an upperbound could help to improve the lower bounds for the following counting function
$$
S^f(x,\alpha):=\left|\Big\{(\epsilon_D,D); 2\leq D\leq x, D \text{ is not a square, and } \epsilon_D\leq D^{\frac{1}{2}+\alpha}\Big\}\right|,
$$
for small values of $\alpha$. In this article, we are interested in a different application of Bourgain's result (see Proposition \ref{B-prop} below) related to squarefree numbers in arithmetic progressions.

Let $X\geq 1$. let $a,q$ be integers, with $q\geq 1$. We let
\begin{equation}\label{E}
E(X,q,a):=\sum_{\substack{n\leq X\\ n\equiv a\!\!\!\pmod q}}\mu^2(n)-\dfrac{6}{\pi^2}\left(1-\frac{1}{q^2}\right)^{-1}\frac{X}{q}.
\end{equation}
For fixed $q$, the last term is known to be equivalent to 
$$
\frac{1}{\phi(q)}\sum_{\substack{n\leq X\\ (n,q)=1}}\mu^2(n)
$$
as $X\rightarrow \infty$. So that $E(X;q,a)$ can be seen as an error term of the distribution of squarefree numbers in arithmetic progressions.
One naturally has the trivial bound

\begin{equation}\label{trivial}
\Big|E(X,q,a)\Big|\leq \dfrac{X}{q} + 1 
\end{equation}

In a previous article, we \cite{RMN} proved

\begin{thm}\label{thmV2}
There exists an absolute constant $C>0$, such that, for every $\epsilon>0$, we have 

\begin{equation}\label{maineq}
\sum_{\substack{a \!\!\!\!\pmod q\\(a,q)=1}}\!\!\!\!\!\!{}^{*}\;E(X,q,a)^2 \sim C\prod_{p\mid q}\bigg(1+2p^{-1}\bigg)^{-1}X^{1/2}q^{1/2},
\end{equation}
for $X\rightarrow \infty$, uniformly for $q$ integer satisfying $X^{31/41+\epsilon}\leq q\leq X^{1-\epsilon}$.
\end{thm}
This theorem gives the asymptotic variance of the above mentioned distribution.\\
Inspired by an equivalent problem considered by Fouvry \textit{et al} \cite[Theorem 1.5.]{FGKM}, we studied how $E(X,q,a)$ correlates with $E(X,q,\gamma(a))$ for suitable choices of $\gamma: \Z/q\Z\rightarrow \Z/q\Z$. It is natural to choose $\gamma$ to be an affine linear map,  \textit{i.e.}

\begin{equation}\label{gamma}
\gamma_{r,s}(a)=ra+s, 
\end{equation}
where $r,s\in \Z$, $r\neq 0$ are fixed. Thus our objet of study is the following correlation sum
\begin{equation}\label{CorFunc}
C[\gamma_{r,s}](X,q):=\sum_{\substack{a \!\!\!\!\pmod q\\ a\neq 0,\gamma_{r,s}^{-1}(0)}}E(X,q,a)E(X,q,\gamma_{r,s}(a)),
\end{equation}
for $q$ prime.
In \cite{RMN}, we already considered the case $s=0$, and we found that correlation always existed for any non zero value of $r$.\\
In particular, there exists $C_r\neq 0$ such that for $X\rightarrow \infty$, $X^{31/41+\epsilon}\leq q\leq X^{1-\epsilon}$, one has

\begin{equation}\label{dep}
 C[\gamma_{r,0}](X,q)\sim C_r\left({\sum_{a\!\!\!\!\pmod q}\!\!\!\!\!}^{*}\;E(X,q,a)^2\right).
\end{equation}
Our main result is the following theorem which exhibits a certain independence between the functions $a\mapsto E(X,q,a)$ and $a\mapsto E(X,q,\gamma_{r,s}(a))$ considered as random variables on $\Z/q\Z$, which confirms our intuition on this question when $\gamma_{r,s}$ is not an homothety.

\begin{thm}\label{main-aa+1}
There exists an absolute $\delta>0$ such that

-for every $\epsilon>0$,

-for every $r$ integer, $r\neq 0$,\\
there exists $C_{\epsilon,r}$ such that one has the inequality

\begin{equation}\label{eq-aa+1}
\Big|C[\gamma_{r,s}](X,q)\Big|\leq C_{\epsilon,r}\left(q^{1+\epsilon}+X^{1/2}q^{1/2}(\log q)^{-\delta}+\dfrac{X^{5/3+\epsilon}}{q}+\left(\dfrac{X}{q}\right)^2\right)
\end{equation}
uniformly for $X\geq 2$, and $q$ prime $\leq X$ such that $q\nmid rs$.
\end{thm}

A consequence of Theorems \ref{thmV2} and \ref{main-aa+1} (not necessarily with the same $\epsilon$) is the following

\begin{cor}
For every $\epsilon>0$ and $r\neq 0$, there exists a function $\Phi_{\epsilon,r}:\R^{+}\rightarrow\R^{+}$, tending to zero at infinity, such that for every $X>1$, for every integer $s$ and for any prime $q$ such that $q\nmid rs$ and $X^{7/9+\epsilon}\leq q\leq X^{1-\epsilon}$, one has the inequality

\begin{equation}\label{indep}
\Big|C[\gamma_{r,s}](X,q)\Big|\leq \Phi_{\epsilon,r}(X)\left({\sum_{a\!\!\!\!\pmod q}\!\!\!\!\!}^{*}\;E(X,q,a)^2\right).
\end{equation}
\end{cor}
Inequality \eqref{indep} shows a behavior different from \eqref{dep} corresponding to $s=0$. In other words, it indicates some independence of the random variables.

Here, as in \cite{RMN}, we give results that are true for a general $r\neq 0$, but in order to simplify the presentation, we give proofs that are only complete when $r$ is squarefree (the case where  $\mu^2(r)=0$ implies a more difficult definition of the $\kappa$ function in \eqref{kappa}).

\section{Notation}

We define the Bernoulli polynomials $B_k(x)$ for $k\geq 1$, on $[0,1)$, in the following recursive way

$$B_1(x):=x-1/2$$
$$\frac{d}{dx}B_{k+1}(x)=B_{k}(x),$$
$$\int_0^1B_k(x)dx=0.$$
We can extend these functions to periodic functions defined in the whole real line by posing
$$
B_k(x):=B_k(\{x\}).
$$
We further notice that $B_1(x)$ satisfy the following relation
\begin{equation}\label{floor-B1}
\lfloor x\rfloor=x-\frac{1}{2}-B_1(x)
\end{equation}
and $B_2(x)$ satisfies
\begin{equation}
B_2(x)=\frac{x^2}{2}-\frac{x}{2}+\frac{1}{12}\text{ for }0\leq x\leq 1.
\end{equation}
In the course of the proof we will make repetitive use of the following multiplicative function
\begin{equation}\label{def-h}
h(d)=\mu^2(d)\prod_{p\mid d}\left(1-2p^{-2}\right)^{-1}.
\end{equation}
We also define here the closely related product
\begin{equation}\label{C2}
 C_2=\prod_p\left(1-\frac{2}{p^2}\right).
\end{equation}

We denote, as usual, by $d(n)$, $d_3(n)$ the classical binary and ternary divisor functions, respectively. We write $\omega(n)$ for the number of primes dividing $n$.
We write $n\sim N$ as an alternative to $N< n\leq 2N$. If $I\subset R$ is an interval, $|I|$ denotes its length.
We use indistinguishably the notations $f=O(g)$ and $f\ll g$ when there is an absolute constant $C$ such that
$$
|f|\leq Cg,
$$
on a certain domain of the variables which will be clear by the context, and the the same for the symbols $O_{\epsilon}$, $O_{r}$, $O_{\epsilon,r}$ and $\ll_{\epsilon}$, $\ll_{r}$, $\ll_{\epsilon,r}$, but with constants that may depend on the subindexed variables.

\section{Initial Steps}\label{initial}
Let $X>1$. Let $\gamma=\gamma_{r,s}$ be given by \eqref{gamma} and let $q$ be a prime number $\leq X$ such that $q\nmid rs$.\\ 
We start by completing the sum defining $C[\gamma](X,q)$ (see \eqref{CorFunc}) and we bound trivially the exceding terms. We have, in view of \eqref{trivial}, that
\begin{equation}\label{C+err}
C[\gamma](X,q)=\sum_{a=0}^{q-1}E(X,q,a)E(X,q,\gamma(a))+O\left(\left(\dfrac{X}{q}\right)^2\right),
\end{equation}
%
In what follows, for simplification, we shall write
\begin{equation}\label{cq}
 C(q)=\dfrac{6}{\pi^2}\left(1-\frac{1}{q^2}\right)^{-1}.
\end{equation}
As we develop the first sum on the right-hand side of \eqref{C+err}, we obtain

\begin{equation}\label{developC}
C[\gamma](X,q)=S[\gamma](X,q)-2C(q)\frac{X}{q}\sum_{n\leq X}\mu^2(n)+C(q)^2\dfrac{X^2}{q}+O\left(\frac{X^2}{q^2}\right),
\end{equation}
where $S[\gamma](X,q)$ is defined by the double sum
\begin{equation}\label{def-S}
S[\gamma](X,q)=\underset{\substack{n_1,n_2\leq X\\n_2\equiv \gamma(n_1)\!\!\!\pmod q}}{\sum\sum}\mu^2(n_1)\mu^2(n_2).
\end{equation}

We point out that $S[\gamma](X,q)$ is the only difficult term appearing in equation \eqref{developC}, since we have the well-known formula

\begin{align}\label{SF-easy}
\sum_{n\leq X}\mu^2(n) &= \dfrac{6}{\pi^2}X +O(\sqrt{X})\notag\\
&=C(q)X +O\left(\dfrac{X}{q^2}+\sqrt{X}\right),
\end{align}
uniformly for $1\leq q \leq X.$ An asymptotic expansion of $S[\gamma](X,q)$ will be given in Proposition \ref{Sfin}.

\section{Useful lemmata}

We start with a lemma concerning the multiplicative function $h(d)$ which is a simple consequence of \cite[lemma 4.2]{RMN}

\begin{lem}\label{conv}
Let $h(d)$ be as in \eqref{def-h} and let $\beta$ be the multiplicative function defined by
$$
h(d)=\sum_{mn=d}\beta(m)\text{, }d\geq 1.
$$
Then $\beta(m)$ satisfies
\begin{align}
\sum_{m\geq M}\frac{\beta(m)}{m}&\ll \frac{(\log 2M)^2}{M}\label{conv2},\\
\sum_{m\leq M}\beta(m)&\ll M\label{conv3},
\end{align}
uniformly for every $M\geq 	1$.
\end{lem}

\begin{proof}
By \cite[lemma 4.2]{RMN}, we know that $\beta(m)$ is supported on cubefree numbers and, if we write $m=ab^2$ with $a$, $b$ squarefree and relatively prime, then
$$
\beta(m)\ll \dfrac{d(a)}{a^2}.
$$
In particular, $\beta(m)\ll 1$, which is sufficient to prove \eqref{conv3}. In order to prove \eqref{conv2}, we notice that
\begin{align*}
\sum_{m\geq M}\frac{\beta(m)}{m}&\ll \underset{ab^2\geq M}{\sum\sum}\dfrac{d(a)}{a^3b^2}\\
&\ll \sum_{n\geq M}\dfrac{d_3(n)}{n^2} \ll \dfrac{(\log 2M)^2}{M}.
\end{align*}
  
\end{proof}

The next proposition is the main result from \cite{Bourgain}, which is crucial to our proof. 
\begin{prop}\label{B-prop}{(see \cite[Proposition 4]{Bourgain})}
There exist constants $c,C,C'$ such that for every $N, q\geq 2$ and $\frac{1}{\log 2N}<\beta<\frac{1}{10}$, there exist a subset $E_N\subset \{1,2,\ldots,N\}$(independent of $q$) satisfying	

\begin{equation}\label{E<}
|E_N|\leq C' \beta\left(\log \frac{1}{\beta}\right)^CN
\end{equation}
and such that, uniformly for $(a,q)=1$, one has
\begin{equation}\label{n-not-in-E}
\left|\sum_{n\leq N, n\not\in E_N, (n,q)=1}e\left(\frac{a{\bar n}^2}{q}\right)\right|\leq C' (\log 2N)^{C}N^{1-c\left(\beta\frac{\log N}{\log q}\right)^{C}}.
\end{equation}

\end{prop}
\begin{rmk}
In the statement of his result, Bourgain uses the symbol $\underset{\sim}{<}$, where one writes $f(x) \underset{\sim}{<} g(x)$ if there is some $C>0$ such that
$$
f(x)\leq Cg(Cx)+C.
$$
In our case, it is easy to see that his result implies Proposition \ref{B-prop}.
\end{rmk}

In fact we specifically need the following corollary

\begin{cor}\label{Cor-Bourgain}
There exists $\delta>0$ such that for every $\epsilon>0$, we have
$$
\sum_{n\leq N, (n,q)=1}e\left(\frac{a{\bar n}^2}{q}\right)\ll_{\epsilon} N(\log q)^{-\delta},
$$
uniformly for $N,q\geq 2$ and $N\geq q^{\epsilon}$.
\end{cor}
\begin{rmk}

More generally, we may consider the sum
$$
\Sigma(I,q)=\sum_{\substack{n\in I\\ (n,q)=1}}e\left(\frac{a{\bar n}^2}{q}\right)
$$
where $I$ is a general interval of length $N \pmod{q}$. By the completion of exponential sums and Weil's bound for such sums, we know that
\begin{equation}\label{Weil}
\Sigma(I,q)\ll q^{1/2}\log q,
\end{equation}
for $q$ prime. Hence, \eqref{Weil} is non trivial as soon as $N\geq q^{\epsilon}$ (for any $\epsilon>1/2$). Obvioulsy, Bourgain's result is much stronger than \eqref{Weil}, but it only applies to intervals containing $0$, roughly speaking.
\end{rmk}

\begin{proof}{(of Corollary \ref{Cor-Bourgain})}
We use Proposition \ref{B-prop} and make the choice $\beta=(\log N)^{-\delta_1}$, where $\delta_1=\min\left(\frac{1}{2},\frac{1}{2C}\right)$. We add together inequalities \eqref{E<} and \eqref{n-not-in-E} to obtain
$$
\sum_{n\leq N, (n,q)=1}e\left(\frac{a{\bar n}^2}{q}\right)\ll N\frac{(\log\log N)^{C}}{(\log N)^{-\delta_1}} + N\frac{(\log N)^C}{\exp(c\epsilon^C(\log N)^{1/2})}.
$$
The corollary now follows by taking, for example, $\delta=\delta_1/2$. 
\end{proof}

\begin{rmk}
Corollary \ref{Cor-Bourgain} will be essential to the proof of Proposition \ref{Sfin}, in which we use it for values of $N$ which are roughly of size $\sqrt{\frac{X}{q}}$. Since we want to take $q$ as large as $X^{1-\epsilon}$, it is very important that Bourgain's result holds for $N$ as small as $q^{\epsilon}.$
\end{rmk}

The next lemma is very similar in essence to many others to be found in literature, for example \cite[Theorem 1]{Tsang}, \cite[Proposition 1.4]{Blomer} or \cite[Theorem 3]{Reuss}. The proof, for instance, follows the lines of \cite[Proposition 1.4]{Blomer}. 
\begin{lem}\label{alatsang}
Let $X>1$ and let $\ell$, $r$ be integers, $r$ squarefree. Let
\begin{equation}
 I(X,\ell,r):=\Big\{u\in \R;u\text{ and }ru+\ell\in (0,X)\Big\}
\end{equation}
and
\begin{equation}\label{Slr}
S(\ell,r):=\sum_{n\in I(X,\ell,r)}\mu^2(n)\mu^2(rn+\ell).
\end{equation}
Then, for every $r>0$, we have the equality
\begin{equation}\label{S(l,r)=}
S(\ell,r)=f(\ell,r)|I(X,\ell,r)|+O_{r}\left(d_3(\ell)X^{2/3}(\log 2X)^{7/3}\right),
\end{equation}
uniformly for $X,\ell\geq 1$.
where
\begin{equation}\label{fq}
 f(\ell,r)=C_2\prod_{p\mid r}\left(\frac{p^2-1}{p^2-2}\right)\prod_{\substack{p^2\mid \ell\\p\nmid r}}\left(\dfrac{p^2-1}{p^2-2}\right)\kappa((\ell,r^2)),
\end{equation}
with
\begin{align}\label{kappa}
\kappa(p^{\alpha})=
\begin{cases}
 \dfrac{p^2-p-1}{p^2-1},& \text{ if } \alpha=1,\\
 \:\:\:\:\dfrac{p^2-p}{p^2-1},& \text{ if } \alpha=2,\\
 \:\:\:\:\:\:\:\:\:\:0,& \text{ if } \alpha\geq 3.
\end{cases}
\end{align}
We recall that $C_2$ and $h(d)$ were already defined in \eqref{C2} and \eqref{def-h} respectively.
\end{lem}

\begin{proof}

We start by defining
$$
\sigma(n)=\prod_{p^2\mid n}p\text{, \;}n\neq 0,
$$
and 
\begin{equation}\label{xi}
\xi(n)=\sigma(n)\sigma(rn+\ell).
\end{equation}
Notice that the right-hand side of equation \eqref{xi} above actually depends on $\ell$ and $r$, but since these numbers will be held fixed in the following calculations, we omit this dependency.\\
Since $\xi(n)$ is an integer $\geq 1$ and since
$$
\mu^2(n)\mu^2(rn+\ell)=1\iff \xi(n)=1,
$$
we deduce the equality
\begin{equation}\label{S=muNd}
S(\ell,r)=\displaystyle\sum_{n\in I(X,\ell,r)}\sum_{d\mid \xi(n)}\mu(d)=\sum_{d\geq 1}\mu(d)N_d(\ell,r),
\end{equation}
where
$$
N_d(\ell,r)=\left|\Big\{n\in I(X,\ell,r); \;\xi(n)\equiv 0\!\!\!\pmod{d}\Big\}\right|.
$$
Notice that the condition
$$
p\mid \xi(n)
$$
only depends on the congruence class of $n\!\! \pmod{p^2}$, for fixed values of $\ell$ and $r$. We let
\begin{equation}\label{u}
u_p(\ell,r):=\left|\Big\{0\leq v\leq p^2-1; \xi(v)\equiv 0 \!\!\!\!\pmod {p}\Big\}\right|,
\end{equation}
and
$$
U_d(\ell,r):=\prod_{p\mid d}u_p(\ell,r).
$$
Then, by the Chinese Remainder Theorem, we have the equality
\begin{equation}\label{Nd=}
N_d(\ell,r)=U_d(\ell,r)\frac{|I(X,\ell,r)|}{d^2}+O\left(U_d(\ell,r)\right),
\end{equation}
for every positive squarefree integer $d$.\\
We also notice that if $(p,r)=1$, then $|u_p(\ell,r)|\leq 2$ and that $|u_p(\ell,r)|\leq p^2$ in general. Therefore we have the upper bound
$$
U_d(\ell,r)\ll_{r} 2^{\omega(d)}.
$$

Let $2\leq y\leq X$ be a parameter, which will be chosen later to be a power of $X$. As we multiply formula \eqref{Nd=} by $\mu(d)$ and sum for $d\leq y$, we obtain the equality
\begin{equation}\label{d<y--1}
\sum_{d\leq y}\mu(d)N_d(\ell,r)=\sum_{d\leq y}\mu(d)U_d(\ell,r)\frac{|I(X,\ell,r)|}{d^2}+O_r\left(\sum_{d\leq y}2^{\omega(d)}\right).
\end{equation}
By completing the first sum on the right-hand side of \eqref{d<y--1}, we have
\begin{equation}\label{d<y}
\sum_{d\leq y}\mu(d)N_d(\ell,r)=\prod_{p}\left(1-\frac{u_p(\ell,r)}{p^2}\right)|I(X,\ell,r)|+O_r\left(\dfrac{X\log y}{y}+y\log y\right). 
\end{equation}
For large values of $d$, formula \eqref{Nd=} is useless. Instead of it we will deduce by different means an estimation for
$$
N_{>y}(\ell,r):=\sum_{d> y}\mu(d)N_d(\ell,r)
$$
from which we will deduce the result.

We notice that $d\mid \xi(n)$ if and only if there exist $j,k$ such that $d=jk$, $j^2\mid n$ and $k^2\mid rn+\ell$. Moreover since $n,rn+\ell<X$, we have $j,k<\sqrt{X}$. From this observation we deduce

\begin{align}\label{first-jk}
\left|N_{>y}(\ell,r)\right|&=\left|\displaystyle\displaystyle\sum_{\substack{y< d\leq X}}\mu(d)\left|\Big\{n\in I(X,\ell,r); \; \xi(n)\equiv 0\!\!\!\pmod d\Big\}\right|\right| \notag\\
&\leq\displaystyle\sum_{\substack{j,k\leq \sqrt{X}\\jk>y}}\left|\Big\{n\in \Z;0<n,rn+\ell<X\text{ and } j^2\mid n, k^2\mid rn+\ell\Big\}\right|\notag\\
&=\displaystyle\sum_{\substack{j,k\leq \sqrt{X}\\jk>y}}N(j,k),
\end{align}
by definition.

We shall divide the possible values of $j$ and $k$ into sets of the form

$$\mathcal{B}(J,K):=\Big\{(j,k); j\sim J,k\sim K\Big\}.$$
We can do the division using at most $O((\log X)^2)$ of these sets, since we are summing over $j, k\leq X^{1/2}$.\\
Let

\begin{multline}\label{NJK}
\mathcal{N}(J,K):= \sum_{j\sim J, k\sim K}N(j,k)\\
=\left|\Big\{(j,k,u,v); j\sim J,k\sim K, 0<j^2u,k^2v<X,\text{ and }k^2v=rj^2u+\ell\Big\}\right|
\end{multline}

By taking the maximum over all $J,K$, we obtain a pair $(J,K)$ with $J,K\leq X^{1/2}$ such that $JK\geq y/4$ and we have the upper bound

\begin{equation}\label{S''<Nlog}
N_{>y}(\ell,r)\ll \mathcal{N}(J,K)(\log X)^2.
\end{equation}
%

At last, we estimate $\mathcal{N}(J,K)$ in the following way

$$
\mathcal{N}(J,K)\leq \sum_{k\sim K}\sum_{u\leq XJ^{-2}}\sum_{\substack{j\sim J\\ j^2ru\equiv -\ell\!\!\!\pmod{k^2}}}1.
$$

For $j,k$ relevant to the sum above, we write $f=(j,k).$ From the congruence condition in the inner sum, we have that $f^2\mid \ell$. So we write
$$
j_0=\frac{j}{f},\;k_0=\frac{k}{f}\text{ and }\ell_0=\frac{\ell}{f^2}.
$$
The congruence then becomes
$$
j_0^2ru\equiv -\ell_0\pmod{k_0^2}.
$$
Now, let $g=(k_0^2,r)$ as above we have $g\mid \ell_0$. We write
$$
k_1=\frac{k_0^2}{g},s=\frac{r}{g}\text{ and }t=\frac{\ell_0}{g}.
$$
That transforms the congruence into
$$
j_0^2su\equiv - t\!\!\!\pmod{k_1}.
$$
Finally, let $h=(k_1,t)$. From the considerations above, we must have $h\mid u$. We write
$$
k'=\frac{k_1}{h},t'=\frac{t}{h}\text{ and }u'=\frac{u}{h}.
$$
So the congruence becomes
$$
j_0^2su'\equiv - t'\!\!\!\pmod{k'}
$$
and since $(t',k')=1$, it has at most $2.2^{\omega(k')}\leq 2d(k_0)$ solutions in $j_0 \pmod{k'}$. Therefore we have

\begin{align*}
\mathcal{N}(J,K)&\leq\sum_{g\mid r}\sum_{f^2h\mid \ell}\sum_{\substack{k_0\sim K/f\\ gh\mid k_0^2}}\sum_{u'\leq XJ^{-2}h^{-1}}\sum_{\substack{j_0\sim J/f\\ j_0^2su'\equiv -t'\!\!\!\pmod{k_0^2/gh}}}1\\
&\leq 2\sum_{g\mid r}\sum_{f^2h\mid \ell}\sum_{k_0\sim K/f}XJ^{-2}h^{-1}\left\{\dfrac{Jgh}{fk_0^2}+1\right\}d(k_0)\\
&\ll_r \sum_{f^2h\mid \ell}\sum_{k_0\sim K/f}XJ^{-2}\left\{\dfrac{J}{fk_0^2}+1\right\}d(k_0)\\
&\ll \sum_{f^2h\mid \ell}XJ^{-2}\left\{\dfrac{J}{K^2}+\frac{1}{f}\right\}K\log K\\
&\ll d_3(\ell)XJ^{-2}\left\{\dfrac{J}{K^2}+1\right\}K\log X.
\end{align*}

Hence
$$
\mathcal{N}(J,K)\ll_r d_3(\ell)\left\{Xy^{-1}+XJ^{-2}K\right\}\log X.
$$
A similar inequality with the roles of $J$ and $K$ interchanged on the right hand side can be obtained in an analogous way. Combining the two formulas, we deduce

\begin{align}\label{NJK-end}
\mathcal{N}(J,K)&\ll_r d_3(\ell)\left\{Xy^{-1}+X{(JK)}^{-1/2}\right\}\log X\notag\\
&\ll d_3(\ell)Xy^{-1/2}\log X. 
\end{align}

Replacing formula \eqref{NJK-end} in \eqref{S''<Nlog} and adding the latter to \eqref{d<y}, it gives
$$
S(\ell,r)=\prod_{p}\left(1-\frac{u_p(\ell,r)}{p^2}\right)|I(X,\ell,r)|+O_r\left(y\log y+d_3(\ell)Xy^{-1/2}(\log X)^3\right).
$$

We make the choice $y=X^{2/3}(\log X)^{4/3}$ obtaining

\begin{equation}\label{almost}
S(\ell,r)=\prod_{p}\left(1-\frac{u_p(\ell,r)}{p^2}\right)|I(X,\ell,r)|+O_r\left(d_3(\ell)X^{2/3}(\log X)^{7/3}\right).
\end{equation}
  
We finish by a study of $u_p(\ell,r)$. We distinguish five different cases (we recall that $r$ is squarefree)

\begin{itemize}
 \item If $p\mid r$, $p^2\mid \ell$ then 
 $$
 u_p(\ell,r)=p,
 $$
 \item If $p\mid r$, $p\mid \ell$ but $p^2\nmid \ell$ then 
 $$
 u_p(\ell,r)=p+1,
 $$
 \item If $p\mid r$, $p\nmid \ell$ then 
 $$
 u_p(\ell,r)=1,
 $$
 \item If $p\nmid r$, $p^2\mid \ell$ then 
 $$
 u_p(\ell,r)=1,
 $$
 \item If $p\nmid r$, $p^2\nmid \ell$ then 
 $$
 u_p(\ell,r)=2.
 $$ 
\end{itemize}

The lemma is now a consequence of formula \eqref{almost} and the different values of $u_p(\ell,r)$.
\end{proof}

\subsection{Sums involving the $B_2$ function}
${}$\\
In the following we study certain sums involving the Bernoulli polynomials $B_2(x)$. In the next lemma, we deal with the simplest case
\begin{equation}\label{A}
A(Y;q,a)=\sum_{\substack{n\geq 1\\(n,q)=1}}\left\{B_2\left(\frac{Y^2}{n^2}+\frac{{a\bar n}^2}{q}\right)-B_2\left(\frac{{a\bar n}^2}{q}\right)\right\},
\end{equation}
where $Y$ is a positive real number, $a,q$ are coprime integers.
The sum above will serve as an archetype for more complicated sums appearing in the proof of Proposition \ref{propsum}, which in their turn will be central for estimating $C[\gamma](X,q)$.
One elementary bound for $A(Y;q,a)$ can be given by noticing that we have both

\begin{equation}\label{B2B}
B_2\left(\frac{Y^2}{n^2}+\frac{{a\bar n}^2}{q}\right)-B_2\left(\frac{{a\bar n}^2}{q}\right)\ll 1, 
\end{equation}
since $B_2$ is bounded, and

\begin{align}\label{B1B}
B_2\left(\frac{Y^2}{n^2}+\frac{{a\bar n}^2}{q}\right)-B_2\left(\frac{{a\bar n}^2}{q}\right)&=\int_{\frac{{a\bar n}^2}{q}}^{\frac{Y^2}{n^2}+\frac{{a\bar n}^2}{q}}B_1(v)dv\notag\\
&\ll \frac{Y^2}{n^2},
\end{align}
since $B_1$ is also a bounded function. Gathering \eqref{B2B} and \eqref{B1B}, we obtain

\begin{align}\label{el.bd}
A(Y;q,a)&\ll \sum_{n\leq Y}1+\sum_{n>Y}\dfrac{Y^2}{n^2}\notag\\
&\ll Y.
\end{align}
In the following lemma we give a non-trivial bound for the sum above by means of Bourgain's bound, in the form of Corollary \ref{Cor-Bourgain}. What we obtain is better than trivial by just a small power of $\log q$, but it is sufficient to obtain Theorem \ref{main-aa+1}.

\begin{lem}\label{exp}
There exists $\delta>0$ such that for every $\epsilon>0$, we have the inequality
\begin{equation}\label{eq-exp}
A(Y;q,a)\ll_{_{\epsilon}} Y(\log q)^{-\delta},
\end{equation}
uniformly for $a$ and $q$ integers satisfying $q\geq 2$ $(a,q)=1$, and $Y>q^{\epsilon}$.
\end{lem}

\begin{proof}
By Corollary \ref{Cor-Bourgain}, we know that there exists $\delta_1>0$ such that
\begin{equation}\label{bourg}
\sum_{\substack{n\leq Y\\(n,q)=1}}e\left(\frac{a{\bar n}^2}{q}\right)\ll_{\epsilon} Y(\log q)^{-\delta_1},
\end{equation}
uniformly for $(a,q)=1$ and $Y>q^{\epsilon/10}$.
For simplification, we write

\begin{equation}\label{Delta}
\Delta_Y\left(n;q,a\right)=B_2\left(\frac{Y^2}{n^2}+\frac{{a\bar n}^2}{q}\right)-B_2\left(\frac{{a\bar n}^2}{q}\right).
\end{equation}

The sum on the left-hand side of \eqref{bourg} appears naturally once we use the Fourier series developpment for $B_2(x)$
\begin{equation}\label{FOURIER}
B_2(x)=\sum_{h\neq 0}\dfrac{1}{4\pi^2 h^2}e(hx)
\end{equation}
in formula \eqref{eq-exp}. Let

\begin{equation}\label{theta=}
 \theta(q)=(\log q)^{\delta_1/2}.
\end{equation}
By \eqref{B2B} and the Fourier decomposition of $B_2(x)$ \eqref{FOURIER}, we have

\begin{align}\label{first-exp}
\sum_{\substack{n\leq Y\theta(q)\\(n,q)=1}}\Delta_Y\left(n;q,a\right)&=\sum_{\substack{Y\theta(q)^{-1}\leq n\leq Y\theta(q)\\(n,q)=1}}\Delta_Y\left(n;q,a\right)+O(Y\theta(q)^{-1})\notag\\
&=\sum_{h\neq 0}\dfrac{1}{4\pi^2h^2}\sum_{\substack{Y\theta(q)^{-1}\leq n\leq Y\theta(q)\\(n,q)=1}}\left(e\left(\dfrac{hY^2}{n^2}\right)-1\right)e\left(\frac{ah{\bar n}^2}{q}\right)+O(Y\theta(q)^{-1})\notag\\
&=\sum_{1\leq |h|\leq \theta(q)^3}\dfrac{1}{4\pi^2h^2}\sum_{\substack{Y\theta(q)^{-1}\leq n\leq Y\theta(q)\\(n,q)=1}}\left(e\left(\dfrac{hY^2}{n^2}\right)-1\right)e\left(\frac{ah{\bar n}^2}{q}\right)+O(Y\theta(q)^{-1}).
\end{align}

Summing by parts, we see that the inner sum of the right-hand side of inequality \eqref{first-exp} is

$$
\ll \sum_{\substack{Y\theta(q)^{-1}\leq m\leq Y\theta(q)}}\dfrac{hY^2}{m^3}\left|\sum_{\substack{Y\theta(q)^{-1}\leq n\leq m\\ (n,q)=1}}e\left(\frac{ah{\bar n}^2}{q}\right)\right|+\left|\sum_{Y\theta(q)^{-1}\leq n\leq Y\theta(q)}e\left(\frac{ah{\bar n}^2}{q}\right)\right|.\notag\\
$$
Now, if $q$ is prime and sufficiently large, then any integer $h$ satisfying $1\leq |h|\leq \theta(q)^3$ is coprime with $q$. Then, by \eqref{bourg}, the above expression is 

\begin{align}\label{Abel-exp}
&\ll \sum_{\substack{Y\theta(q)^{-1}\leq m\leq Y\theta(q)}}\dfrac{|h|Y^2}{m^2}(\log q)^{-\delta_1}+Y\theta(q)^{-1}\notag\\
&\ll |h|Y\theta(q)^{-1}.
\end{align}
As we insert the upper-bound \eqref{Abel-exp} in formula \eqref{first-exp}, we obtain

\begin{equation}\label{exp-n-small}
\sum_{\substack{n\leq Y\theta(q)\\(n,q)=1}}\Delta_Y\left(n;q,a\right)\ll Y\theta(q)^{-1}\log\log q\ll Y(\log q)^{-\delta_1/4}.
\end{equation}

For the remainder terms we use the trivial upper bound \eqref{B1B} to deduce the inequality

\begin{align}\label{exp-n-big}
\sum_{\substack{n> Y\theta(q)\\(n,q)=1}}\Delta_Y\left(n;q,a\right)&\ll \sum_{n>Y\theta(q)}\dfrac{Y^2}{n^2}\notag\\
&\ll Y\theta(q)^{-1}.
\end{align}
We combine the upper bounds \eqref{exp-n-small} and \eqref{exp-n-big} to conclude. Together they give
$$
\sum_{\substack{n\geq 1\\(n,q)=1}}\Delta_Y\left(n;q,a\right)\ll Y(\log q)^{-\delta_1/4}.
$$
uniformly for $(a,q)=1$ and $Y>q^{\epsilon}$.  The proof of lemma \ref{exp} is now complete.
\end{proof}
\begin{rmk}
Among the hypothesis of lemma \eqref{exp}, it is essential that we have $(a,q)=1$. In the case where $q\mid a$, one can not improve on \eqref{el.bd}. Indeed, it is possible to show that (see \cite[lemma 4.3]{RMN})
$$
A(Y;q,0)=-\dfrac{\varphi(q)}{q}\dfrac{\zeta(3/2)}{2\pi}Y+O(d(q)Y^{2/3})\;\;\;(Y\geq 1).
$$
\end{rmk}

\subsection{A consequence of Lemma \ref{exp}}
${}$\\
In order to evaluate $S[\gamma](X,q)$ (see \eqref{def-S}), it is important to consider the following sum which appears in equation \eqref{S(l,r)=}

\begin{dfn}\label{fraksgamma}
For $q,r,s$ integers satisfying $q\geq 1$, $q\nmid rs$, let
\begin{equation}\label{frakS}
\mathfrak{S}[\gamma](X,q):=\sum_{\ell\equiv s\!\!\!\pmod q}f(\ell,r)|I(X,\ell,r)|,
\end{equation}
where $\gamma=\gamma_{r,s}$.
\end{dfn}

The purpose of this subsection is to prove the following
\begin{prop}\label{propsum}
There exists $\delta>0$, such that for every $\epsilon>0$, for every $r\neq 0$ squarefree, one has
\begin{equation}\label{propsum-eq}
\mathfrak{S}[\gamma](X,q)=\left(\dfrac{6}{\pi^2}\right)^2\left(1+\dfrac{1}{q^2(q^2-2)}\right)^{-1}X^2/q+O_{\epsilon,r}(q^{1+\epsilon} + X^{1/2}q^{1/2}(\log q)^{-\delta}),
\end{equation}
uniformly for $X>1$, $s$ integer and $q$ prime such that $q\nmid rs$, with $C(q)$ as in \eqref{cq}

\end{prop}
The special case $r=1$ simplifies many of the calculations in the proof below. For instance, the sums over $\rho$, $\sigma$ and $\tau$ disappear. Although, this simpler result is, in fact, equally deep and it shows more clearly the connection between the upper bound \eqref{eq-exp} and the error term in \eqref{propsum-eq} 
\begin{proof}
We start by recalling \eqref{fq}

$$
f(\ell,r)=C_2\prod_{p\mid r}\left(\frac{p^2-1}{p^2-2}\right)\prod_{\substack{p^2\mid \ell\\p\nmid r}}\left(\frac{p^2-1}{p^2-2}\right)\kappa((\ell,r^2)),
$$
where $C_2$ is as in \eqref{C2}. We notice that the first and second terms on the right-hand side of equation above are independent of $\ell$, that means that in order to evaluate $\mathfrak{S}[\gamma](X,q)$, we need to study

$$
\mathfrak{S}'[\gamma](X,q)=\sum_{\ell\equiv s\!\!\!\pmod q}|I(X,\ell,r)|\prod_{\substack{p^2\mid \ell\\p\nmid r}}\left(\frac{p^2-1}{p^2-2}\right)\kappa((\ell,r^2)).
$$
i.e.
\begin{equation}\label{S',S}
\mathfrak{S}'[\gamma](X,q)=C_2^{-1}\prod_{p\mid r}\left(\frac{p^2-1}{p^2-2}\right)^{-1}\mathfrak{S}[\gamma](X,q).
\end{equation}

We expand the product $\displaystyle\prod_{\substack{p^2\mid \ell\\p\nmid r}}\left(\frac{p^2-1}{p^2-2}\right)$ as
$$
\prod_{\substack{p^2\mid \ell\\p\nmid r}}\left(\frac{p^2-1}{p^2-2}\right)=\sum_{\substack{d^2\mid \ell\\(d,r)=1}}\dfrac{h(d)}{d^2},
$$
from which we deduce

\begin{align}\label{propsum-first}
\mathfrak{S}'[\gamma](X,q)&:=\sum_{\rho\mid r^2}\kappa(\rho)\sum_{\substack{\ell\equiv s\!\!\!\!\!\pmod{q}\\(\ell,r^2)=\rho}}|I(X,\ell,r)|\sum_{\substack{d^2\mid \ell\\(d,r)=1}}\dfrac{h(d)}{d^2}\notag\\
&=\sum_{\rho\sigma\mid r^2}\kappa(\rho)\mu(\sigma)\sum_{\ell_0\equiv {\overline{\rho\sigma}}s\!\!\!\!\!\pmod{q}}|I\left(X,\rho\sigma\ell_0,r\right)|\sum_{\substack{d^2\mid \ell_0\\(d,r)=1}}\dfrac{h(d)}{d^2}\notag\\
&=\sum_{\rho\sigma\mid r^2}\kappa(\rho)\mu(\sigma)\sum_{(d,qr)=1}\dfrac{h(d)}{d^2}\sum_{\ell_1\equiv  \overline{(\rho\sigma d^2)}s\!\!\!\!\!\pmod{q}}\left|I(X,\rho\sigma d^2\ell_1,r)\right|
\end{align}
where in the second line we used M\"obius inversion formula for detecting the gcd condition and we noticed that the congruence satisfied by $\ell_0$ implies $(d,q)=1$.

We write the inner sum as an integral:

\begin{equation}\label{Y->int}
\sum_{\ell_1\equiv  \overline{(\rho\sigma d^2)}s\!\!\!\!\!\pmod{q}}\left|I(X,\rho\sigma d^2\ell_1,r)\right|=\int_{0}^{X}\sum_{\ell_1\equiv  \overline{(\rho\sigma d^2)}s\!\!\!\!\!\pmod{q}}\mathbf{1}_{(0,X)}(ru+\rho\sigma d^2\ell_1)du,\\
\end{equation}
where $\mathbf{1}_{(0,X)}$ is the characteristic function of the interval $(0,X)$. Hence the inner sum above equals

\begin{multline*}
\left\lfloor\dfrac{X-ru}{\rho\sigma d^2q}-\frac{\overline{(\rho\sigma d^2)}s}{q}\right\rfloor-\left\lfloor\dfrac{-ru}{\rho\sigma d^2q}-\frac{\overline{(\rho\sigma d^2)}s}{q}\right\rfloor\\
= \frac{X}{\rho\sigma d^2q}-B_1\left(\dfrac{X-ru}{\rho\sigma d^2q}-\frac{\overline{(\rho\sigma d^2)}s}{q}\right)+B_1\left(\dfrac{-ru}{\rho\sigma d^2q}-\frac{\overline{(\rho\sigma d^2)}s}{q}\right),
\end{multline*}
for almost all $u\in (0,X)$ in the sense of Lebesgue measure.\\
If we use this formula  in equation \eqref{Y->int}, we deduce the equality

\begin{multline}\label{B2s}
\sum_{\ell_1\equiv  \overline{(\rho\sigma d^2)}s\!\!\!\!\!\pmod{q}}\left|I(X,\rho\sigma d^2\ell_1,r)\right|=\\
\frac{X^2}{\rho\sigma d^2q}-\frac{\rho\sigma d^2q}{r}\left\{B_2\left(\dfrac{X^2}{\rho\sigma d^2q}-\frac{\overline{(\rho\sigma d^2)}s}{q}\right)-B_2\left(-\frac{\overline{(\rho\sigma d^2)}s}{q}\right)\right.\\
\left.-B_2\left(\dfrac{(1-r)X}{\rho\sigma d^2q}-\frac{\overline{(\rho\sigma d^2)}s}{q}\right)+B_2\left(\frac{-rX}{\rho\sigma d^2q}-\frac{\overline{(\rho\sigma d^2)}s}{q}\right)\right\}.
\end{multline}
From this point on, we suppose $r<0$. The case $r>0$ requires only minor modifications. With this hypothesis, we have that both
$$
\frac{(1-r)X}{\rho\sigma d^2q}\text{ and }\dfrac{-rX}{\rho\sigma d^2q}
$$
are positive for every $\rho,\sigma\geq 1$.

We inject \eqref{B2s} above in equation \eqref{propsum-first} and we define
$$
B(D;q,a;r):=\sum_{(d,qr)=1}h(d)\Delta_D(d,q;a),
$$
where $\Delta_D(d,q;a)$ is as in \eqref{Delta}. From \eqref{propsum-first} and \eqref{B2s} we deduce the equality

\begin{multline}\label{frak'=lambda+3G}
\mathfrak{S}'[\gamma](X,q)=\lambda(q,r)\frac{X^2}{q}-\frac{q}{r}\left\{ G\left(\frac{X}{q};q,-s;r\right)\right.\\
\left.-G\left(\frac{(1-r)X}{q};q,-s;r\right)+G\left(\frac{-rX}{q};q,-s;r\right)\right\},
\end{multline}
where
\begin{equation*}
G(Y;q,s;r)=\underset{\rho\sigma\mid r^2}{\sum\sum}\kappa(\rho)\mu(\sigma)\rho\sigma B\left(\sqrt{\frac{Y}{\rho\sigma}},q,\overline{\rho\sigma}s;r\right),
\end{equation*}
and
\begin{equation*}
\lambda(q,r)=\sum_{\rho\sigma\mid r^2}\frac{\kappa(\rho)\mu(\sigma)}{\rho\sigma} \times \sum_{(d,qr)=1}\dfrac{h(d)}{d^4}.
\end{equation*}
Returning to the function $\beta(m)$ defined in Lemma \ref{conv}, we observe that for a general $D>0$, one has

\begin{align*}
B\left(D;q,a;r\right)&=\sum_{(m,qr)=1}\beta(m)\sum_{(n,qr)=1}\Delta_{D}(mn;q,a) \\
&=\sum_{(m,qr)=1}\beta(m)\sum_{(n,qr)=1}\Delta_{D/m}(n;q,\overline{m}^2a)\\
&=\sum_{(m,qr)=1}\beta(m)\sum_{\tau\mid r}\mu(\tau)\sum_{(n,q)=1}\Delta_{D/\tau m}(n;q,{\overline{\tau}}^2{\overline{m}}^2a)\\
&=\sum_{(m,qr)=1}\beta(m)\sum_{\tau\mid r}\mu(\tau)A(D/\tau m,q;{\overline{\tau}}^2{\overline{m}}^2a).
\end{align*}

We apply the equality above with  $D=\sqrt{\frac{Y}{\rho\sigma}}$ and $a=\overline{\rho\sigma}s$, multiply by $\kappa(\rho)\mu(\sigma)\rho\sigma$ and sum over $\rho,\sigma$ such that $\rho\sigma\mid r^2$, we have

\begin{equation}\label{quad}
G(Y;q,s;r)=\underset{\rho\sigma\mid r^2}{\sum\sum}\sum_{\tau\mid r}\sum_{(m,qr)=1}\kappa(\rho)\mu(\sigma)\mu(\tau)\rho\sigma\beta(m)A\left({\sqrt{\frac{Y}{\rho\sigma\tau^2 m^2}}};q,\overline{\rho\sigma\tau^2 m^2}s\right).
\end{equation}
Our discussion depends on the size of $Y$.

- If $Y\leq q^{\epsilon}$, we have the trivial bound (see \eqref{el.bd})
$$
A\left({\sqrt{\frac{Y}{\rho\sigma\tau^2 m^2}}};q,\overline{\rho\sigma\tau^2 m^2}s\right)\ll {\sqrt{\frac{Y}{\rho\sigma\tau^2 m^2}}}\leq \frac{Y^{1/2}}{m},
$$
for every $\rho, \sigma, \tau\geq 1$. Summing over $\rho, \sigma, \tau$ and $m$, it gives

\begin{align}\label{Ysmall}
G(Y;q,s;r)&\ll_r Y^{1/2}\sum_{m\geq 1}\frac{\beta(m)}{m}\notag\\
&\ll q^{\epsilon/2},
\end{align}
as a consequence of upper bound \eqref{conv2}.

- If $Y>q^{\epsilon}$, we separate the quadruple sum on the right-hand side of \eqref{quad} as

$$
\underset{\substack{m\leq q^{\epsilon/2}\\ \rho\sigma\tau^2m^2> Y/q^{\epsilon}}}{\sum\sum\sum\sum}+\underset{\substack{m\leq q^{\epsilon/2}\\ \rho\sigma\tau^2m^2\leq Y/q^{\epsilon}}}{\sum\sum\sum\sum}+\underset{m>q^{\epsilon/2}}{\sum\sum\sum\sum}.
$$
For the first sum we have, again, the trivial bound

\begin{equation}\label{primeira}
A\left({\sqrt{\frac{Y}{\rho\sigma\tau^2 m^2}}};q,\overline{\rho\sigma\tau^2 m^2}s\right)\ll {\sqrt{\frac{Y}{\rho\sigma\tau^2 m^2}}}\leq q^{\epsilon/2},
\end{equation}
The most delicate sum is the second one, since we appeal to \eqref{eq-exp}. This gives

\begin{equation}\label{segunda}
A\left({\sqrt{\frac{Y}{\rho\sigma\tau^2 m^2}}};q,\overline{\rho\sigma\tau^2 m^2}s\right)\ll_{\epsilon} {\sqrt{\frac{Y}{\rho\sigma\tau^2 m^2}}}(\log q)^{-\delta}.
\end{equation}
For the third one, we use the trivial bound,

\begin{equation}\label{terceira}
A\left({\sqrt{\frac{Y}{\rho\sigma\tau^2 m^2}}};q,\overline{\rho\sigma\tau^2 m^2}s\right)\ll {\sqrt{\frac{Y}{\rho\sigma\tau^2 m^2}}},
\end{equation}
Gathering the inequalities \eqref{primeira}, \eqref{segunda} and \eqref{terceira} in \eqref{quad}, we obtain

$$
G(Y;q,s;r)\ll_{\epsilon,r} q^{\epsilon/2}\sum_{m\leq q^{\epsilon/2}}|\beta(m)|+\sqrt{Y}(\log q)^{-\delta}\sum_{m\leq q^{\epsilon/2}}\frac{|\beta(m)|}{m}+\sqrt{Y}\sum_{m>q^{\epsilon/2}}\frac{|\beta(m)|}{m},
$$
and finally, by Lemma \ref{conv}

\begin{equation}\label{Ybig}
G(Y;q,s;r)\ll_{\epsilon,r} q^{\epsilon} + \sqrt{Y}(\log q)^{-\delta}\text{\; \; \;}(Y>q^{\epsilon}).
\end{equation}
Comparing with \eqref{Ysmall}, we have that \eqref{Ybig} is true for any $Y\geq 1$.

Combining \eqref{Ybig} and \eqref{frak'=lambda+3G}, one has
\begin{equation}\label{almost-frakS}
\mathfrak{S}'[\gamma](X,q)=\lambda(q,r)\frac{X^2}{q}+O_{\epsilon,r}(q^{1+\epsilon} + X^{1/2}q^{1/2}(\log q)^{-\delta}).
\end{equation}
If we multiply the formula above by $C_2\displaystyle\prod_{p\mid r}\left(\frac{p^2-1}{p^2-2}\right)$ (recall formula \eqref{S',S}), we deduce

\begin{equation}\label{end}
\mathfrak{S}[\gamma](X,q)=\Lambda(q,r)\frac{X^2}{q}+O_{\epsilon,r}(q^{1+\epsilon} + X^{1/2}q^{1/2}(\log q)^{-\delta}),
\end{equation}
where

$$
\Lambda(q,r)=C_2\prod_{p\mid r}\left(\frac{p^2-1}{p^2-2}\right)\sum_{\rho\sigma\mid r^2}\frac{\kappa(\rho)\mu(\sigma)}{\rho\sigma} \times \sum_{(d,qr)=1}\dfrac{h(d)}{d^4}
$$
Since for $r$ squarefree, we have the equality
$$
\sum_{\rho\sigma\mid r^2}\frac{\kappa(\rho)\mu(\sigma)}{\rho\sigma}=\prod_{p\mid r}\left(\dfrac{p^2-1}{p^2}\right),
$$
then, by some standard calculations, we notice that $\Lambda(q,r)$ does not depend on $r$. More precisely, since, $q$ is prime and $(q,r)=1$, we have

$$
\Lambda(q,r)=\left(\dfrac{6}{\pi^2}\right)^2\left(1+\dfrac{1}{q^2(q^2-2)}\right)^{-1}.
$$
As a consequence, formula \eqref{end} completes the proof of Proposition \ref{propsum}. 
\end{proof}

\section{Study of $S[\gamma](X,q)$}

We rewrite $S[\gamma](X,q)$ (see \eqref{def-S}) as
\begin{equation}\label{S=nh}
S[\gamma](X,q)=\sum_{\substack{\ell\equiv s\!\!\!\!\!\!\pmod q}}\sum_{n\in I(X,\ell,r)}\mu^2(n)\mu^2(rn+\ell).
\end{equation}
First we notice that the inner sum equals zero if $|\ell|>2|r|X$. Hence, by formula \eqref{S(l,r)=}, we have that

\begin{equation}\label{eq-ts}
S[\gamma](X,q)=\sum_{\substack{\ell\equiv s\!\!\!\pmod q\\|\ell|\leq 2|r|X}}f(\ell,r)|I(X,\ell,r)|+O_r\left(\dfrac{X}{q}X^{2/3+\epsilon}\right),
\end{equation}
for $X\geq q$. We notice that if $|\ell|>2|r|X$, one also has $|I(X,\ell,r)|=0$, hence we can complete the sum on the right-hand side of \eqref{eq-ts}. Thus, we can write (recall definition \eqref{frakS})

$$
S[\gamma](X,q)=\mathfrak{S}[\gamma](X,q)+O_r\left(\dfrac{X^{5/3+\epsilon}}{q}\right).
$$
From Proposition \ref{propsum}, we deduce the equality

$$
S[\gamma](X,q)=\left(\dfrac{6}{\pi^2}\right)^2\left(1+\dfrac{1}{q^2(q^2-2)}\right)^{-1}\dfrac{X^2}{q}+O_{\epsilon,r}\left(q^{1+\epsilon}+X^{1/2}q^{1/2}(\log q)^{-\delta}+\dfrac{X^{5/3+\epsilon}}{q}\right).
$$
In view of the definition \ref{cq} of $C(q)$, it is easy to see that
\begin{equation*}
\left(\dfrac{6}{\pi^2}\right)^2\left(1+\dfrac{1}{q^2(q^2-2)}\right)^{-1}=C(q)^2+O\left(\frac{1}{q^2}\right).
\end{equation*}
In conclusion, we proved
\begin{prop}\label{Sfin}
There exists $\delta>0$ such that for every $\epsilon>0$ and every $r\neq 0$, one has the asymptotic formula
\begin{equation}\label{S-final}
S[\gamma_{r,s}](X,q)=C(q)^2\dfrac{X^2}{q}+O_{\epsilon,r}\left(q^{1+\epsilon}+X^{1/2}q^{1/2}(\log q)^{-\delta}+\dfrac{X^{5/3+\epsilon}}{q}+\dfrac{X^2}{q^3}\right),
\end{equation}
uniformly for $X\geq 2$, for every integer $s$ and for any prime $q$ such that $q\nmid rs$ and $q\leq X$.
\end{prop}

\section{Proof of the main Theorem}

We start by recalling the formula \eqref{developC}
$$
C[\gamma](X,q)=S[\gamma](X,q)-2C(q)\frac{X}{q}\sum_{n\leq X}\mu^2(n)+C(q)^2\dfrac{X^2}{q}+O\left(\frac{X^2}{q^2}\right).
$$
By Proposition \ref{Sfin} and formula \eqref{SF-easy}, we directly obtain the equality

\begin{equation*}
C[\gamma](X,q)=O_{\epsilon,r}\left(q^{1+\epsilon}+X^{1/2}q^{1/2}(\log q)^{-\delta}+\dfrac{X^{5/3+\epsilon}}{q}+\dfrac{X^2}{q^2}\right).
\end{equation*}
The proof of Theorem \ref{main-aa+1} is now complete.

\end{document}